\documentclass[a4paper,10pt,draft]{article}
\usepackage{amsfonts}
\usepackage[american]{babel}
\newtheorem{theorem}{Theorem}

\newtheorem{definition}[theorem]{Definition}

\newtheorem{lemma}[theorem]{Lemma}

\newtheorem{proposition}[theorem]{Proposition}
\newtheorem{remark}[theorem]{Remark}

\newenvironment{proof}[1][Proof]{\noindent\textbf{#1.} }{\ \rule{0.5em}{0.5em}}

\def\D{{\cal D}}

\def\eps{\varepsilon}
\def\B{{\cal B}}

\def\A{{\cal A}}

\def\C{{\cal C}}

\def\R{{\mathbb R}}

\def\E{{{\rm I} \kern -.15em {\rm E}    }}

\newcommand{\reals}{{{\rm I} \kern -.15em {\rm R} }}
\newcommand{\nat}{{{\rm I} \kern -.15em {\rm N} }}

\usepackage[usenames]{color}
\definecolor{red}{rgb}{1.0,0.0,0.0}

\definecolor{blu}{rgb}{0.0,0.0,1.0}

\definecolor{gre}{rgb}{0.03,0.50,0.03}

\input{tcilatex}

\begin{document}

\begin{titlepage}
\vspace{-2truecm}
\begin{center}
    \large
   \bfseries
   On the equivalence of internal and external habit formation models with finite memory \footnote{We thank Herakles Polemarchakis for several useful discussions on this topic.} \\[2\baselineskip]
    \normalfont
   Emmanuelle Augeraud-Veron \\
	{\small University of La Rochelle}\\ [1\baselineskip]
   Mauro Bambi \\
    {\small University of York} \\ [1\baselineskip]
    Fausto Gozzi \\
    {\small LUISS and University of Pisa}
 \\[2\baselineskip]
\vspace{-0.5cm}
 {26 March 2014} \\[4\baselineskip]
 \bfseries
\vspace{-1.5cm}

	Abstract \\ [2\baselineskip]
\vspace{-1.25cm}
\end{center}
	
\normalfont

\noindent In this paper we use a dynamic programming approach to analytically solve an endogenous growth model with internal habits where the key parameters describing their formation, namely the intensity, persistence and lag structure (or memory), are kept generic. Then we show that external and internal habits lead to the same closed loop policy function and then to the same (Pareto) optimal equilibrium path of the aggregate variables when the utility function is subtractive nonseparable. The paper uses new theoretical results from those previously developed by the dynamic programming literature applied to optimal control problems with delay and it extends the existing results on the equivalence between models with internal and external habits to the case of finite memory.

\vspace{.1cm}
 \vspace{.1cm}


  \end{titlepage}

\setcounter {footnote} {0}

\newpage

\tableofcontents

\section{Introduction}

\textit{Motivation and Results --}
In this paper we consider the simplest endogenous growth model with linear technology, as in Rebelo \cite{rebe}, and we assume that the representative household's utility function depends also on internal habits, whose formation is based on the history (up to a given fixed lag $\tau$) of past consumption. The resulting instantaneous utility is
\[
u(c(t),h(t))
\]
while the habit formation is described by the following exponentially
smoothed index of the past consumption rates
\begin{equation}
h(t)=\varepsilon \int_{t-\tau }^{t} c(u)e^{\eta (u-t)}du \qquad \forall
t\geq 0  \label{HABITEX}
\end{equation}
with $\varepsilon\geq0$, $\eta\geq0$, and $\tau\geq0$ indicating
respectively the intensity, persistence and lags structure (or memory) of the habits. The habit formation equation (\ref{HABITEX}) is general in its assumptions on the intensity, persistence and lag structure  and it embeds all the main specifications used in the literature; the role of $\tau$ is indeed critical in pinning down the different forms of the habits (e.g. $\tau=1$ is the continuous time version of the case studied by Boldrin et al. \cite{Bol-Chr-Fis} among others).\footnote{Equation (\ref{HABITEX}) does not include the case with deep habits studied by Raven et al. \cite{Ravn} since we focus on a single consumption good economy.} A generic and finite choice of the memory parameter $\tau$ is also consistent with recent empirical evidences.\footnote{Among them, Crawford \cite{Craw} uses a
revealed preference approach to characterize the internal habits. He finds no sharp
result on the lag structure: increasing the number of period lags in the consumption
of the good increases the ``agreement between theory and data. However it (...) has a large
negative effect on the power of the test compared with the one-lag version''. In his contribution, he looks at the cases $\tau=1,2$ and $3$.}

Our objective is to solve analytically this problem with internal habits using a dynamic programming approach and to prove that, independently on the choice of the habits formation's parameters, the solution of this problem coincides with the solution of the same problem but with external habits
when the instantaneous utility function has the nonseparable
subtractive form:\footnote{This instantaneous utility function is, together with the multiplicative nonseparable,  one of the two most common specifications used in the habit formation literature.}
\begin{equation}
u(c(t),h(t))=\frac{(c(t)-h(t))^{1-\gamma }}{1-\gamma } \qquad \gamma>0, \ \gamma \neq 1. \label{subU}
\end{equation}

To arrive to this result we prove that,
keeping all the else equal, the problem with internal habits leads to the
same solution path of its counterpart with external habits; in the latter
the instantaneous utility function has the same functional form but the habits are now formed over the past average economy-wide consumption, $\bar c(\cdot)$:

\[
h(t)=\varepsilon \int_{t-\tau }^{t} \bar c(u)e^{\eta (u-t)}du \qquad \forall
t\geq 0;
\]

Our contribution is relevant for the following two main reasons. Firstly, it provides a full analytical characterization of an endogenous growth model with internal habits and finite memory. To achieve this result we have  extended the dynamic programming approach to optimal control of Delay Differential Equations (DDE)
first developed in Fabbri and Gozzi \cite{FaGo} to a different framework. In fact, the presence of the habit formation with a potentially finite lag parameter, $\tau$, implies a substantial analytical deviation from other problems studied in the literature since here the delay is contained in the objective function. Therefore, an extension of the previous results on dynamic programming approach to optimal control of DDE's is necessary to find explicitly the policy function of our problem: in this extent, our paper represents a new contribution to the dynamic programming literature in infinite dimension, as it will be extensively explained later in this introduction and in Section 3. Secondly, it extends the existing results on consumption externalities not leading to economic distortions. More precisely, we prove that the equivalence holds when the three key parameters $\varepsilon$, $\eta$, and $\tau$ in the habit formation equation are kept generic, while in previous contributions were assumed either $\varepsilon=\eta$ and $\tau=1$ (e.g. Alonso-Carrera et al. \cite{AlCab}) or $\varepsilon=\eta$ and $\tau=\infty$ (e.g. Gomez \cite{Gom}).


\bigskip

\bigskip

\textit{Methodology --}
To prove our results we use a dynamic programming approach to find the solution of the model with internal habits, then we compare it with the solution in the case with external habits and finally we show that the two resulting closed loop policy functions are identical when the utility function has the nonseparable
subtractive form.

The model with external habits was solved in Augeraud-Veron and Bambi
\cite{auba} using a modified version of the Pontryagin Maximum Principle (PMP) (see e.g. \cite{OksendalPMP}); the closed
loop policy function was also found from the explicit computation of capital,
consumption and habits.
Such PMP approach has been recently used by several
authors to solve vintage capital models
(e.g. Barucci and Gozzi \cite{BG1}, \cite{BG2},  Boucekkine et al. \cite{Bouc3}, \cite{Bouc4}, \cite{Bouc5},
Br\'echet et al. \cite{BTV10}
Feichtinger et al \cite{FJET}, \cite{FJMAA}, Saglam and Veliov \cite{SaglamVeliov08}
Veliov \cite{Veliov08}) and time to build models
(e.g. Bambi \cite{bambi} and Bambi and Gori \cite{BaGo}).

The same strategy can be applied to the case with
internal habits
but it won't lead to an explicit formulation of the optimal policy
because of the mixed type equation resulting from the PMP in presence of retarded control.
So we proceed to solve the problem with internal habits and we
find
the closed loop policy function through the dynamic programming method;
this approach successfully leads to identify the explicit form of the closed loop policy function as soon as its associated
Hamilton-Jacobi-Bellman equation (HJB) can be solved explicitly.
It must be noted that the delayed structure of the problem pins down an HJB equation which is a partial differential equation in infinite dimension without explicit solutions unless specific assumptions on the production and utility function are
introduced. Luckily enough, the linear production function and the
nonseparable subtractive form of the utility function let us develop an ad
hoc approach in order to calculate explicitly the solutions of the HJB
equation
and then the closed loop policy functions which, as explained
before, are crucial to prove the equivalence between the internal and
external habit formation model.
\\
The dynamic programming approach to
optimal control problems with delay has found very few applications
in the economic literature.
As far as we know the first to apply this method were Fabbri and
Gozzi \cite{FaGo} in a vintage capital framework, and later Boucekkine et
al. \cite{Bouc1} and Bambi et al. \cite{bambi2010}, the latter in a
time-to-build model (see also \cite{BaGo}, \cite{FagGozMPS} \cite{FagGozJME} and \cite{FagSICON}
for application of the same technique to models with age structure).
More recently Boucekkine et al. \cite{Bouc2} used it to
investigate the compatibility of the optimal population size concepts
produced by different social welfare functions and egalitarianism.
We must note that, from a technical point of view, the problem we face
in this paper is quite different from those in the papers just quoted
because the delayed control appears both in the objective functional
and in the constraints: then we have to extend the theoretical results, already used in the previously cited papers,
to a different context, see on this Remark \ref{rm:technical}.

\bigskip

\textit{Plan of the paper --}
The paper is organized as follows. Section 2 presents the general model with habit formation
where subsection 2.1 is devoted to further explain the case with internal habits.
Section 3 explains how the problem can be rewritten in infinite dimension and how to arrive to the solution path using the Hamilton-Jacobi-Bellman equation.
Section 4 states and proves the equivalence result. Finally Section 5 concludes the paper.

\section{The model\label{sec:thedetcase}}


Consider a standard neoclassical growth model, where the economy consists of
a continuum of identical infinitely lived atomistic households, and firms.
The households' objective is to maximize over time the discounted
instantaneous utility (here $c(t)$ and $h(t)$ are, respectively, the consumption
and the habit at time $t$):
\begin{equation}
u(c(t),h(t))=\frac{(c(t)-h(t))^{1-\gamma }}{1-\gamma },  \label{UU}
\end{equation}
for $c(t)\geq h(t)$ and $\gamma >0$ $\gamma \ne 1$\footnote{The case $\gamma=1$ can be treated exactly as the other ones. We do not do it here to make the analytical part less cumbersome.}. If $c(t)<h(t)$ the utility function is
not always well defined in the real field and it is never concave. For this
reason, it is generally assumed that $u(c(t),h(t))=-\infty $ as soon as
$c(t)<h(t)$. The instantaneous utility function (\ref{UU}) clearly implies
addiction in the habits since current consumption is forced to remain higher
than the habits over time. The habits are formed according to the rule
\begin{equation}
h(t)=\varepsilon \int_{t-\tau }^{t}\hat c(u)e^{\eta (u-t)}du
\qquad \forall t\geq 0  \label{HABIT}
\end{equation}
where $\hat c(t)$ indicates the customary consumption level, which is equal
to $c(t)$ in the case of internal habits or to the economy-wide average
consumption, $\bar c(t)$, when we consider external habits. Moreover $\eta >0$ measures the persistence of habits, while $\varepsilon > 0$ the
intensity of habits, i.e. the importance of the economy average consumption
relative to current consumption. Finally the habits' past history, $h(t)$
with $t\in[-\tau,0)$ is given; also, in the case of external habits,
the path of $\bar c(t)$ is taken as given
since no individual decision have an appreciable effect on the average
consumption of the economy.

Differentiating (\ref{HABIT}) (this is possible e.g. in all continuity points of $\hat c(\cdot)$)
we have
\begin{equation}
\dot{h}(t)=\varepsilon \left( \hat c(t)- \hat c(t-\tau )e^{-\eta \tau
}\right) -\eta h(t)\qquad \qquad \forall t\geq 0  \label{HABITdiff}
\end{equation}


Assuming a linear technology $y(t)=Ak(t)$ and a depreciation factor $\delta>0$
the optimal control problem to be solved is
\begin{eqnarray*}
&&\displaystyle\max \int_{0}^{\infty }
\frac{\left(c(t)-\varepsilon \int_{t-\tau }^{t}\hat{c}
(u)e^{\eta (u-t)}du \right)^{1-\gamma }}{1-\gamma }\,e^{-\rho t}dt \\
&&s.t.\ \ \dot{k}(t)=\left( A-\delta \right) k(t)-c(t) \\
&&\qquad k(t)\geq 0,\ c(t)\geq 0, \  c(t) \ge\varepsilon \int_{t-\tau }^{t}\hat{c}%
(u)e^{\eta (u-t)}du\\
&&\qquad k(0)=k_{0}>0, \qquad \hat c(u) \;\; given\; for \; u\in [-\tau,0).
\end{eqnarray*}
Observe that our framework describes implicitly an economy where new capital is financed by a riskless technology whose instantaneous rate of return is $A-\delta$. It is also worth noting that in the case $\hat{c}(t)=\bar{c}(t)$ the habits
enters as an externality in the utility function. On the other hand, when
$\hat{c}(t)=c(t)$ the habits are formed over the representative agent's past
consumption level and then are internalized in her resource allocation
decision. Observe also that in the last case the social welfare theorems
guarantee that the solution is Pareto optimal while in the former the
presence of the externality could lead to an inefficient path. In the rest
of the paper we will study the problem with internal habits ($\hat{c}(t)=c(t)$)
and then we will compare our findings with those (presented in \cite{auba})
of the model with external habits ($\hat{c}(t)=\bar{c}(t)$).

\subsection{The internal habit problem}

The social planner problem consists in finding the strategy
$c\left( \cdot\right) $ which maximizes the objective functional
\begin{equation}\label{eq:objective}
\int_{0}^{\infty }\frac{\left(c(t)-\varepsilon \int_{t-\tau }^{t}c(u)e^{\eta
(u-t)}du \right)^{1-\gamma }}{1-\gamma }\;e^{-\rho t}dt
\end{equation}
under the state equation (which can be seen as an equality constraint)
\begin{equation}
\dot{k}(t)=\left( A-\delta \right) k(t)-c(t), \quad t \ge 0,  \label{eq:stateeq}
\end{equation}
the positivity constraints
\begin{equation}\label{eq:constraints1}
k(t)\geq 0,\qquad c(t)\geq 0,
\end{equation}
the constraint
\begin{equation}\label{eq:constraints2}
c(t) \ge \varepsilon
\int_{t-\tau }^{t}c(u)e^{\eta (u-t)}du,
\end{equation}
and with initial data $k(0)=k_{0}>0$, $c\left( s\right) =c_{0}\left( s\right) $
given for $s\in \lbrack -\tau ,0)$, where we assume
$c_{0}\left( \cdot\right)
\in L^{1}\left( [ -\tau ,0);\mathbb{R}^{+}\right)$.\footnote{$L^{1}\left( [ -\tau ,0);\mathbb{R}^{+}\right)$ indicates the space of functions from $[ -\tau ,0)$ to $\mathbb{R}^{+}$ which are Lebesgue measurable and integrable.} Observe that we have already substituted in (\ref{eq:objective}) and (\ref{eq:constraints2}) the equation describing the internal habit formations:
\begin{equation}
h(t)=\varepsilon \int_{t-\tau }^{t} c(u)e^{\eta (u-t)}du
\qquad \forall t\geq 0.  \label{HABITnew}
\end{equation}
whose initial value is known and is given by
\begin{equation}
h_{0}:=\varepsilon \int_{-\tau }^{0}c_0(u)e^{\eta (u-t)}du.
\label{eq:initialdatah}
\end{equation}


In the following two sections we solve the social planner
problem by using the dynamic programming approach.
Then, in Section \ref{sec:equivalence}, we prove
our equivalence theorem showing that the
closed loop policy formula found in Section \ref{sec:HJB}
is the same found in \cite{auba} for the market equilibrium problem.

\section{Solution of the internal habit problem\label{sec:solution}}

\subsection{Preliminary results\label{sec:preliminary}}


We first introduce a notation useful to rewrite more formally equation
(\ref{HABITnew}) and the objective functional. We call $c_{0}(\cdot):\left[
-\tau ,0\right) \rightarrow
\mathbb{R}_{+}$ the initial datum, $c(\cdot):[0,\infty )\rightarrow
\mathbb{R}_{+}$ the control strategy and
$\widetilde{c}:[-\tau ,\infty )\rightarrow
\mathbb{R}_{+}$ the function (sometimes called the {\em concatenation}
of the two above)
\[
\widetilde{c}\left( s\right) =\left\{
\begin{array}{c}
c_{0}\left( s\right) \text{ for }s\in \lbrack -\tau ,0) \\
c\left( s\right) \text{ for }s\in \lbrack 0,\infty )
\end{array}
\right.
\]
With this notation equation (\ref{HABITnew}) is rewritten more precisely
as
\begin{equation}
h(t)=\varepsilon \int_{t-\tau }^{t} \widetilde{c}(u)e^{\eta (u-t)}du
\qquad \forall t\geq 0.  \label{HABITtilde}
\end{equation}
The state equation (\ref{eq:stateeq}) is a standard linear
Ordinary Differential Equation (ODE) and so
it can be easily seen that, for every locally integrable control strategy
$c\left( \cdot\right) :\mathbb{R}_{+}\rightarrow\mathbb{R}_+$\footnote{The space of such functions will be denoted from now on by $L^1_{loc}([0,+\infty);\R_+)$.},
there exists a unique absolutely continuous solution
of it, which will be denoted as $k_{k_{0},c(\cdot)}\left( \cdot\right) $
and which is given by
\begin{equation}
\label{eq:DDE}
k\left( t\right) =k_{0}e^{\left( A-\delta \right) t}-\int_{0}^{t}e^{\left(
A-\delta \right) \left( t-u\right) }{c}(u)du.
\end{equation}
The objective functional to maximize is
$$J\left(  k_{0},c_{0}(\cdot) ;c(\cdot)\right)
:=\int_{0}^{\infty }\frac{\left(c(t)-\varepsilon \int_{-\tau }^{0}\widetilde{c}
(u+t)e^{\eta u}du \right)^{1-\gamma }}{1-\gamma }\;e^{-\rho t}dt,$$
over the set
\[
\mathcal{C}\left( k_{0},c_{0}(\cdot)\right) =\left\{ c\left( \cdot\right) \in
L_{loc}^{1}\left(  [0,+\infty); \R^+ \right):\; k_{k_{0},c(\cdot)}\left( \cdot\right)\ge 0
\right.
\]
\[
\left. \text{ and} \; c(t)\geq \varepsilon
\int_{t-\tau }^{t}\widetilde{c}(u)e^{\eta (u-t)}du\geq 0\text{ for almost every }t\in
\mathbb{R}_{+}\right\}
\]
We call from now on (\textbf{P}) the problem of finding an optimal control strategy
i.e. a strategy
$c^{\ast }(\cdot)\in \mathcal{C}\left( k_{0},c_{0}(\cdot)\right)$
such that
\[
-\infty <J\left(  k_{0},c_{0}(\cdot);c^{\ast }(\cdot)\right)<+\infty
\] and
\[
J\left(  k_{0},c_{0}(\cdot);c^{\ast }(\cdot)\right) =\sup_{c(\cdot)
\in \mathcal{C}\left( k_{0},c_{0}(\cdot)\right)}
\int_{0}^{\infty }\frac{\left(c(t)-\varepsilon
\int_{-\tau }^{0}\widetilde{c}(u+t)e^{\eta u}du\right)^{1-\gamma }}{1-\gamma }
\;e^{-\rho t}dt
\]
As usual we call value function the map
$$V\left( k_{0},c_{0}(\cdot)\right):=
\sup_{c(\cdot)
\in \mathcal{C}\left( k_{0},c_{0}(\cdot)\right)}
\int_{0}^{\infty }\frac{\left(c(t)-\varepsilon
\int_{-\tau }^{0}\widetilde{c}(u+t)e^{\eta u}du\right)^{1-\gamma }}{1-\gamma }
\;e^{-\rho t}dt
$$

We now give a preliminary study of the problem concerning the behavior of admissible trajectories and the finiteness of the value function.


First of all we give an estimate for the admissible control strategies and state trajectories.

\begin{proposition}[Lower bound for admissible strategies]
\label{pr:estimateck}
We consider any initial datum
$\left( k_{0},c_{0}\left( \cdot\right) \right)
\in \left(\mathbb{R}_{+}\times L^{1}\left( [ -\tau ,0) ,
\mathbb{R}_{+}\right) \right) $ and any control strategy
$c\left( \cdot\right) \ge 0 $ satisfying (\ref{eq:constraints2}).
Then we have, for every $t \ge 0$,
\begin{equation}\label{eq:controlnew1}
c(t) \ge c^m(t)
\end{equation}
where
$c^{m}\left( \cdot\right)\in L^1_{loc}([0,+\infty); \R^+ )$
is the unique solution of the equation
\begin{equation}\label{eq:controlnew}
c^{m}\left( t\right)
 =\varepsilon \int_{t-\tau}^{t}\widetilde{c}^{\, m}\left( u\right) e^{\eta \left( u-t\right) }du
\end{equation}
Moreover the state trajectory $k\left( \cdot\right) $ associated to $c(\cdot)$
is dominated at any time $t\geq 0$ by the solution $k^{M}\left( \cdot\right) $
obtained taking the same initial datum $k_{0}$ and the control $c^{m}\left( \cdot\right)$
\begin{equation}\label{eq:statenew}
k(t)\le k^{M}\left( t\right) =e^{\left( A-\delta \right) t}\left[ k_{0}-
\int_{0}^{t}c^{m}\left( u\right) e^{-\left( A-\delta \right) u}du\right].
\end{equation}
\end{proposition}

\begin{proof}
First we observe that, thanks to standard existence theorems for DDE's (see e.g. \cite{HaleIntro}, Section 2.2) the equation (\ref{eq:controlnew}) has a unique solution for every $c_0(\cdot)\in L^1([-\tau,0);\R^+)$.

Now take a control strategy $c(\cdot) \in\mathcal{C}\left( k_{0},c_{0}(\cdot)\right)$. The constraint (\ref{eq:constraints2}) together with (\ref{eq:controlnew}) implies that
$$
c(t)- c^{m}\left( t\right)
\ge\varepsilon \int_{t-\tau}^{t}
 \left[\widetilde{c}\left( u\right)-\widetilde{c}^{m}\left( u\right)\right] e^{\eta \left( u-t\right) }du, \qquad t \ge 0
$$
Clearly, since both functions $c(\cdot)$ and $c^m(\cdot)$ have the same past $c_0(\cdot)$, it must be $\widetilde{c}\left( t\right)-\widetilde{c}^{m}\left( t\right)= 0$ for $t\in [-\tau,0)$.
So, calling $c_1(t):= c(t)- c^{m}\left( t\right)$ we get, for $t \in [0,\tau]$,
$$
c_1(t)\ge \int_0^t c_1(u)e^{\eta \left( u-t\right) }du.
$$
This implies, by a simple application of Gronwall inequality (see e.g. \cite{HaleIntro}[p.15, Lemma 3.1], that $c_1(t) \ge 0$ for $t\in [0,\tau]$.
Take now $t \in (\tau,2\tau]$. As above we have, for any such $t$,
$$
c_1(t)\ge \int_{t-\tau}^{\tau} c_1(u)e^{\eta \left( u-t\right) }du
+\int_\tau^t c_1(u)e^{\eta \left( u-t\right) }du.
$$
Since the function $t \to  \int_{t-\tau}^{\tau} c_1(u)e^{\eta \left( u-t\right) }du$
is nonnegative for every $t \in (\tau,2\tau]$, then applying again the Gronwall inequality we get that $c_1(t) \ge 0$ for $t \in (\tau,2\tau]$.
The claim (\ref{eq:controlnew1}) for every $t\ge 0$ then easily follows by induction.
Finally the claim (\ref{eq:statenew}) follows by (\ref{eq:controlnew1}) and
by the formula (\ref{eq:DDE}).
\end{proof}

\bigskip


The characteristic equation associated to the delay equation (\ref{eq:controlnew}) writes
\begin{equation}\label{eq:carnew}
1=\varepsilon \int_{-\tau }^{0}e^{\left( \lambda +\eta \right) u}du
\end{equation}

\begin{proposition} [Properties of characteristic roots]
\label{pr:careq}
The characteristic equation (\ref{eq:carnew}) admits a unique real root,
$\lambda_0$. We have $\lambda_0<\varepsilon - \eta$
and all complex roots have a real part smaller than $\lambda_0.$
Moreover,
\begin{itemize}
  \item if $1-\varepsilon \int_{-\tau }^{0}e^{\eta u}du<0$ then $\lambda_0$
is the only root with positive real part;
    \item if $1-\varepsilon \int_{-\tau}^{0}e^{\eta u}du>0$ all the roots have negative real part.
    \item if $1-\varepsilon \int_{-\tau}^{0}e^{\eta u}du=0$ then $\lambda_0=0$
     and the other roots have negative real part.
\end{itemize}
\end{proposition}

\begin{proof}
We first study real roots. Consider the function
$$\varphi: \R \to \R, \qquad \varphi \left( \lambda \right)
=1-\varepsilon \int_{-\tau }^{0}e^{\left( \lambda +\eta \right) u}du.
$$
Since $\varphi'(\lambda)=
-\varepsilon\int_{-\tau }^{0}u e^{\left( \lambda +\eta \right) u}du>0$,
then $\varphi$ is a strictly increasing function of $\lambda$.
Moreover
$$\lim_{\lambda \rightarrow -\infty
}\varphi \left( \lambda \right) =-\infty,   \qquad\lim_{\lambda \rightarrow
+\infty }\varphi \left( \lambda \right) =1$$
and
$$\varphi(0)=1-\varepsilon \int_{-\tau}^{0}e^{\eta u}du \ge 1-\varepsilon \tau,
\qquad \varphi(\varepsilon-\eta)=e^{-\varepsilon \tau}>0.
$$
The above equation implies that there exists a unique real root of
the equation $\varphi (\lambda)=0$.
Such root belongs to $(-\infty,\varepsilon - \eta)$.

Moreover, all complex roots $\lambda =p+iq$ satisfy the system
$$
1-\varepsilon \int_{-\tau }^{0}e^{\left( p +\eta \right) u} \cos (qu)du=0
$$
$$
\varepsilon \int_{-\tau }^{0}e^{\left( p +\eta \right) u} \sin(qu)du=0
$$
From the first equation we get that
\[
1<\varepsilon \int_{-\tau }^{0}e^{\left( p+\eta \right) u}du =1- \varphi (p).
\]
This implies that $\varphi (p)<0$ which implies $p< \lambda_0$.

Finally, let $1-\varepsilon \int_{-\tau}^{0}e^{\eta u}du<0$ and consider
the function
$$a \left( \lambda \right) :=(\lambda+ \eta) \varphi
\left( \lambda \right).$$
It can be easily seen that $a \left( \lambda
\right) $ rewrites
\[
a \left( \lambda \right) =\lambda +\eta -\varepsilon \left(
1-e^{-\left( \lambda +\eta \right) \tau }\right)
\]
and that all complex roots of the characteristic equation $\varphi(\lambda)=0$ are also zeros of $a(\cdot)$.
Let us assume that there exists a complex root, $\lambda =p+iq$
with $p \in (0,\lambda_0).$ Then, we have
\[
\func{Re}\left( a\left(\lambda \right) \right)=
p+\eta -\varepsilon +\varepsilon e^{-\left( p+\eta \right) \tau }\cos \left(
q\tau \right) <p+\eta -\varepsilon +\varepsilon e^{-\left( p+\eta \right)
\tau }= a(p)<0
\]
which contradict the fact that $\func{Re}\left( a \left( \lambda
\right) \right) >0.$
The rest of the claim is immediate.
\end{proof}

%

\bigskip

Now we have, as a consequence, the following result.

\begin{proposition}[Existence of admissible paths]
\label{prop:kM}
\begin{itemize}
\item[]
  \item[(i)]
Fix an initial datum $\left( k_{0},c_{0}\left( \cdot\right)
\right) \in \left(\mathbb{R}_{+}\times
L^{1}\left( [ -\tau ,0) ;\mathbb{R}_{+}\right) \right) $.
The set $\mathcal{C}\left( k_{0},c_{0}(\cdot)\right) $
is nonempty if and only if the control
$c^{m}\left( \cdot\right)$ introduced in (\ref{eq:controlnew}) is admissible, i.e. such that $k^M(t)\ge 0$
for every $t \ge 0$.
  \item[(ii)]
In particular, if $\lambda_0\ge A-\delta$ then for any
$c_0(\cdot)\in L^{1}\left( [ -\tau ,0) ;\mathbb{R}_{+}\right)  $,
such that $c_0(t)> 0$ on a set of positive Lebesgue measure
we have
$\mathcal{C}\left( k_{0},c_{0}(\cdot)\right) = \emptyset$.
\end{itemize}
\end{proposition}

\begin{proof}
The first statement is an immediate corollary of Proposition
\ref{pr:estimateck}.

Concerning the second statement we observe
first that the solution
of the equation (\ref{eq:controlnew}) can be written with a series expansion
(see e.g. Corollary 6.4, p.168 of \cite{Diekmann})
as follows
\begin{equation}
\tilde c^m(t) =\sum\limits_{r=0}^{\infty}p_r (t)
e^{\lambda_{r}t} \label{SolDFDE}
\end{equation}
where $\{\lambda_r\}_{r \in \mathbb{N}}$ is the sequence of the roots of
the characteristic equation (\ref{eq:carnew}) and the $p_r (t)$ are
polynomials of degree less or equal to $m(r)-1$ where $m(r)$ is the multiplicity of $\lambda_r$.
Now, using e.g. \cite{Bel-Coo}, Section 6.7 (in particular Theorem 6.5)
we can explicitly compute the coefficients of such solutions
by using the Laplace transform.

In particular, since $\lambda_0$ is a simple root, we have
$$
p_0=\frac{\psi (\lambda_{0})}{\varphi^{\prime}(\lambda_{0})}
$$
where
$$
\varphi(\lambda)=1- \eps\int_{-\tau}^0 e^{(\lambda+\eta)u} du
$$
and
$$\psi (\lambda)= (1-\varphi(\lambda))\int_{-\tau}^0 c_0(u)e^{-\lambda u} du
$$
Clearly, if $c_0(\cdot)>0$ on a set of positive Lebesgue measure
we have that $p_0>0$ and so the leading term of the series (\ref{SolDFDE})
is $p_0e^{\lambda_0 t}$ and all the others are complex
exponentials with negative real part.
So the corresponding state trajectory $k_M(\cdot)$ is
$$
k^{M}\left( t\right) =e^{\left( A-\delta \right) t}\left[ k_{0}-
\int_{0}^{t}p_0 e^{\left(\lambda_0 -(A-\delta) \right) u}du + \xi (t)\right]
$$
where $\xi(\cdot): [0,+\infty) \to \R$ is a bounded function
coming from the lower order term of the series (\ref{SolDFDE}).
When $\lambda_0\ne A - \delta$ it follows
$$
k^{M}\left( t\right) =e^{\left( A-\delta \right) t}
\left[ k_{0}+\frac{p_0}{\lambda_0-(A-\delta)} + \xi(t) \right]
-\frac{p_0}{\lambda_0-(A-\delta)} e^{\lambda_0 t}
$$
Clearly, when $\lambda_0>A-\delta$ the limit of the above expression
is $-\infty$, so the claim follows.
When  $\lambda_0\ne A - \delta$ we have
$$
k^{M}\left( t\right) =e^{\left( A-\delta \right) t}\left[ k_{0}-
p_0 t+ \xi (t)\right]
$$
and again the limit of the above expression
is $-\infty$, so the claim follows.
\end{proof}

%
%

\bigskip

Due to the above proposition it makes sense to study the social planner problem
when
\begin{equation}\label{eq:hpgrowth}
\lambda_0 <A-\delta.
\end{equation}
and the initial datum $c_0(\cdot)$ is small enough so to guarantee
that the corresponding $k_M(\cdot)$ is always strictly positive.
It is clear that $\lambda_0$ is the lowest possible growth rate of the habit: this growth rate has to be lower than the real interest rate of the economy ($r=A-\delta$), which coincides with the maximum growth rate of capital obtainable from the capital accumulation equation when consumption is set to zero. In fact an economy cannot sustain over time a growth rate which exceeds the real interest rate because capital does not accumulate sufficiently fast to sustain the higher and higher consumption.

Note in particular that (\ref{eq:hpgrowth}) is surely true if
\begin{equation}\label{eq:hpgrowthbis}
\varepsilon - \eta \le A-\delta.
\end{equation}
or, since $\frac{\varepsilon}{ \eta}(1-e^{-\eta \tau})<1 \Leftrightarrow \lambda_0<0$, if
\begin{equation}\label{eq:hpgrowthter}
\frac{\varepsilon}{ \eta}(1-e^{-\eta \tau})<1 \qquad and \qquad A-\delta>0.
\end{equation}
From now on, we will focus on the case
\begin{equation}\label{eq:hpgrowthfinal}
\lambda_0<\eps - \eta \le 0 < A-\delta.
\end{equation}
The condition $\eps - \eta \le 0$ is usually assumed in the economic literature (e.g. Constantinides \cite{const}) because it prevents the economy to asymptotically converge to the corner solution $c(t)=h(t)$.

We finally observe that strict positivity of $k_M(\cdot)$ is guaranteed by assuming that, beyond (\ref{eq:hpgrowth})
\begin{equation}\label{eq:hpconstrsuff}
k_0 > \int_0^{+\infty}e^{-s(A-\delta)}c^m(s)ds
\end{equation}
where $c^m(\cdot)$ is the unique solution of (\ref{eq:controlnew}). The economic intuition behind this restriction on the initial condition of capital will be explained in Section \ref{sec:equivalence}.

\bigskip


Therefore conditions (\ref{eq:hpgrowth}) and (\ref{eq:hpconstrsuff}) are necessary to guarantee that the value function $V$ is not always $-\infty$ at a given point. Here we give a sufficient condition for the finiteness of $V$.

%

\begin{proposition}[Finiteness of the value function $V$]
Let us consider an initial datum $\left( k_{0},c_{0}\left( \cdot\right)
\right) \in \left(\mathbb{R}_{+}\times
L^{1}\left( [ -\tau ,0) ;\mathbb{R}_{+}\right) \right) $.
Assume that (\ref{eq:hpgrowth}) and (\ref{eq:hpconstrsuff}) hold true, so
$\mathcal{C}\left( k_{0},c_{0}(\cdot)\right) \ne \emptyset$.
If
\begin{equation}\label{eq:Vfinite}
\rho >(A-\delta ) \left( 1-\gamma \right),
\end{equation}
then the value function is always finite.
\end{proposition}

\begin{proof}
To prove the claim it is enough to prove the following:
\begin{itemize}
\item[(i)] If $\gamma \in (0,1)$ then there exists $M_+>0$ such that, for all
$\left(k_{0},c_{0}\left(\cdot\right) \right)$ in the space $\left(
\mathbb{R}_{+}\times L^{1}\left( [ -\tau ,0) ,\mathbb{R}_{+}\right) \right)$,
  \[
0 \le V\left( k_{0},c_{0}\right) \le M_+ k_0 ^{ 1-\gamma  }.
\]
\item[(ii)] If $\gamma \in (1,+\infty)$ and (\ref{eq:hpgrowthter}) holds, then there exists $M_-<0$ such that, for all
$\left(k_{0},c_{0}\left(\cdot\right) \right)$ in the space $\left(
\mathbb{R}_{+}\times L^{1}\left([ -\tau ,0) ,\mathbb{R}_{+}\right) \right)$,
  \[
M_- k_0 ^{ 1-\gamma  } \le V\left( k_{0},c_{0}\right) \le 0.
\]
\end{itemize}

We prove first (i).
The first inequality is obvious since
for $\gamma \in (0,1)$ we always have $J\left( k_{0},c_{0}(\cdot) ;c(\cdot)\right) \ge 0$.

Concerning the other inequality setting (Fleming and Soner \cite
{fleming}, p.30-32, Freni et al. \cite{freni})
\[
\zeta\left( s\right) =\int_{0}^{s}c\left( u \right) ^{1-\gamma }du
\]
and applying H\"{o}lder's inequality to $\zeta\left( s\right)
=\int_{0}^{s}s^{1-\gamma }\left( \frac{c\left( u \right) }{s}\right)
^{1-\gamma }du $ yields to
\begin{eqnarray*}
\zeta\left( s\right) &\leq &\left( \int_{0}^{s}s^{\frac{1-\gamma }{\gamma }
}du \right) ^{\gamma }\left( \int_{0}^{s}\left( \frac{c\left( u
\right) }{s}\right) ^{\frac{1-\gamma }{1-\gamma }}du \right) ^{1-\gamma }
\\
&\leq &s^{\gamma }\left( \int_{0}^{s}c\left( u \right) du \right)
^{1-\gamma }
\end{eqnarray*}
as $c\left( u \right) =\left( A-\delta \right) k\left( u \right) -\dot{
k}\left( u \right) $
\[
\int_{0}^{s}c\left( u\right) du =\int_{0}^{s}\left( A-\delta \right)
k\left( u \right) du -k\left( s\right) +k\left( 0\right)
\]%
Now, according to equation (\ref{eq:DDE}), $k\left( s\right) \leq k\left(
0\right) e^{\left( A-\delta \right) s}$. Thus using also that $k(s) \ge 0$ for $s \ge 0$
we get
\begin{equation}\label{eq:estimateintc}
\int_{0}^{s}c\left( u
\right) du \leq k_0 e^{\left( A-\delta \right) s}
\end{equation}
and so
\[
\zeta \left( s\right) \leq s^{\gamma }
k_0^{1-\gamma }e^{\left( 1-\gamma \right)
\left( A-\delta \right) s}
\]
Now we have
$$
J\left(  k_{0},c_{0}(\cdot);c(\cdot)\right)  \leq
\int_{0}^{+\infty}\frac{c(s)^{1-\gamma }}{1-\gamma}
e^{-\rho  s}ds
$$
and, integrating by parts and using (\ref{eq:estimateintc}),
\[
J\left(  k_{0},c_{0}(\cdot);c(\cdot)\right)  \leq
\left( \frac{k_0 ^{ 1-\gamma  }}{1-\gamma }
\int_{0}^{+\infty}s^{\gamma }e^{\left( \left( 1-\gamma \right) \left( A-\delta
\right) -\rho \right) s}ds\right)
\]
which gives the claim.

\medskip

Now we prove (ii).
The second inequality is obvious since
for $\gamma \in (1,+\infty)$ we always have $J\left( k_{0},c_{0}(\cdot) ;c(\cdot)\right) \le 0$.

Concerning the other inequality we observe that, calling $c^m(\cdot)$ the unique solution of (\ref{eq:controlnew}) we have, thanks to (\ref{eq:hpconstrsuff}) and (\ref{eq:hpgrowthter}) that,
for $\alpha>0$ small enough, the control strategy defined as
$c_1(t)=c^m(t) + \alpha $ ($t \ge 0$) is admissible.

Indeed, calling $k_1(\cdot)$ the associated state trajectory we have
$$
k_1(t)=e^{\left( A-\delta \right) t}\left[ k_{0}-
\int_{0}^{t} e^{-s\left(A-\delta \right) u} (c^m(u) + \alpha )du\right]
=
$$
$$=
e^{\left( A-\delta \right) t}\left[ k_{0}-
\int_{0}^{t} e^{-s\left(A-\delta \right) u} c^m(u) du
-\frac{\alpha}{A-\delta}
\right]
$$
This remain always positive if
$$\frac{\alpha}{A-\delta}\le k_{0}-
\int_{0}^{+\infty} e^{-s\left(A-\delta \right) u} c^m(u) du
$$
which is possible by (\ref{eq:hpconstrsuff}).
Moreover the control $c_1(\cdot)$ satisfy the constraint (\ref{eq:constraints2})
since, substituting it into (\ref{eq:controlnew}) we get
$$
\alpha \ge \alpha \frac{\eps(1- e^{-\eta \tau})}{\eta}
$$
which is always true for positive $\alpha$ thanks to (\ref{eq:hpgrowthter}).

Since $c_1(\cdot)$ is admissible we have
$$
V(k_0,c_0(\cdot))\ge J(k_0,c_0(\cdot); c_1(\cdot))= \frac{\alpha^{1-\gamma}}{\rho(1-\gamma)}
$$
Now it is clear from what said above that it must be $\alpha\le (A-\delta)k_0$, so the claim follows taking $M_-=\frac{(A-\delta)^{1-\gamma}}{\rho(1-\gamma)}$.
\end{proof}

\bigskip

Observe that condition (\ref{eq:Vfinite})
is the same condition which guarantees bounded utility in a standard AK model.
Therefore habits formation does not affect this condition.

Condition (\ref{eq:Vfinite}) will be assumed from now on without repeating it.

%
%

\subsection{The equivalent infinite dimensional problem\label{sec:HJB}}

We now rewrite our problem as an optimal control problem for ODE's in an infinite dimensional space. Note that, differently from what has been done in the previous literature
(see e.g. \cite{FaGo}) here the state equation (\ref{eq:stateeq}) is not a DDE so the past of the control does not appear there.
The past of the control strategy appears in the objective functional (\ref{eq:objective})  and in the constraint (\ref{eq:constraints2}). For this reason the way we choose to rewrite our problem is different from the one given in the previous literature.

We work in Hilbert space
$M^{2}=\mathbb{R}\times L^{2}\left( \left[ -\tau ,0\right);\R \right)$, with the scalar
product defined by
\[
\left\langle \left( x_{0},x_{1}\left( \cdot\right) \right) ,\left(
y_{0},y_{1}\left( \cdot\right) \right) \right\rangle_{M^{2}}
=x_{0}y_{0}+\int_{-\tau }^{0}x_{1}\left( s\right) y_{1}\left(
s\right) ds
\]
for every $x=\left( x_{0},x_{1}\left( \cdot\right) \right) $ and $y=\left(
y_{0},y_{1}\left( \cdot\right) \right) $ in $M^{2}$.

We first define, following e.g. \cite{VinterKwong}, the structural state of the
infinite dimensional system we want to study.
\begin{definition}[Structural state]
Given an initial datum $(k_{0},c_{0}(\cdot))\in \mathbb{R}\times
L^{1}\left( \left[ -\tau ,0\right);\R \right)$,
and a control strategy $c\left( \cdot\right) \in
L_{loc}^{1}\left(\left[ 0,+\infty \right) ;\mathbb{R}\right) $
we define the structural state of our controlled dynamical system at time $t\geq 0$
as the element of $M^2$:
\[
X_{\left( k_{0},c_{0}\left( \cdot\right)\right) ,c\left( \cdot\right)  }\left(
t\right) =\left( k_{ k_{0},c\left( \cdot\right)
 }\left( t\right) ,s\mapsto \varepsilon \int_{-\tau }^{s}\tilde c\left(
t+u-s\right) e^{\eta u}du\right)
\]
\end{definition}
In the following we would write $X(t)$ for $X_{\left( k_{0},c_{0}\left( \cdot\right)\right)
,c\left( \cdot\right) }(t)$ when no confusion is possible. The second component
of $X(t)$ is a function of $s\in [-\tau,0)$ and we will usually write $X_1(t)[s]$
when we mean its value at time $t$ for given $s\in [-\tau,0)$.

We now define the unbounded operator $\A$ on $M^{2}$ by
$$D\left( \A \right) =\left\{ \left( x_{0},x_{1}\left( \cdot\right)
\right) \in M^{2},\; x_{1}\left( \cdot\right) \in W^{1,2}\left( \left[ -\tau ,0\right] ;
\mathbb{R}\right) ,\, x_{1}\left( -\tau \right) =0\right\} $$
\[
\A x=\left( \left( A-\delta \right) x_{0},-x'_{1}(\cdot)\right)
\]
Moreover we define the operators
$$\B : \mathbb{R} \rightarrow M^{2}, \qquad
\B c=c\left( -1,s\mapsto \varepsilon e^{\eta s}\right)
$$
and
$$\D :\R \times C([-\tau,0];\R) \subset M^{2}\rightarrow \mathbb{R}, \qquad
\D x=x_{1}\left( 0\right).
$$
Now we show that the structural state above satisfy a suitable ODE in the space $M^2$.
\begin{theorem}
Given any initial datum
$\left( k_{0},c_{0}\left( \cdot\right) \right) \in
\mathbb{R}\times L^{1}\left( \left[ -\tau ,0\right);\R \right)$
and any control strategy
$c\left( \cdot\right) \in L_{loc}^{1}\left( \left[ 0,\infty \right) ;\mathbb{R}
\right) $ the associated structural state is the unique solution of the equation
\begin{equation}\label{eq:ss1}
\left\{
\begin{array}{l}
\frac{dX\left( t\right) }{dt}= \A X\left( t\right) +\B c\left( t\right) \\
\\
X\left( 0\right) =\left( k_{0},s\mapsto \varepsilon \int_{-\tau
}^{s}c_{0}\left( u-s\right) e^{\eta u}du\right).
\end{array}
\right.
\end{equation}
\end{theorem}

\begin{proof}
The proof easily follows by the definition of the structural state
and of the operators $\A$ and $\B$.
Uniqueness of the solution is similar to Bensoussan et al.
(\cite{bensousan92}, Theorem 5.1, p.282).
\end{proof}

\bigskip

Now we consider the ODE (\ref{eq:ss1}) with generic initial datum $x\in M^2$
and call $X(t;x,c(\cdot))$ (or simply $X(t)$ when clear from the context)
the unique solution of it for a given
control strategy $c(\cdot)\in L^1_{loc}([0,+\infty);\R_+)$.
We take (\ref{eq:ss1}) as state equation and write a control problem equivalent
to our problem ({\bf{P}}).

The constraint $c(t)\geq \varepsilon \int_{t-\tau }^{t}c(u)e^{\eta (u-t)}du$ writes
\[
c\left( t\right) \geq X_{1}\left( t\right)[0]={\cal D}X(t)
\]
So the set of admissible control strategies for a given initial datum in $x\in M^{2}$
is given by
\[
\C_{ad}\left( x\right) =\left\{ c\left( \cdot\right) \in L_{loc}^{1}\left( \left[
0,\infty \right) ;\mathbb{R}\right) ,\text{ such that } X_0(t) \ge 0,\,
c(t) \ge 0, \, c\left( t\right) \geq X_{1}(t)[0]
\text{ for all }t\right\}
\]

The functional to be maximized becomes
\[
J_{0}\left( x\,;c\left( \cdot\right) \right) :=\int_{0}^{\infty }\frac{
(c(t)-\D X(t))^{1-\gamma }}{1-\gamma }e^{-\rho t}dt
\]
The value function is defined as
\[
V_{0}\left( x\right) :=\max_{c\in \C_{ad}\left( q\right) }J_{0}\left( x\,;c\left(
\cdot\right) \right)
\]
where we set $V_{0}\left( x\right) =-\infty $ if $\C_{ad}\left( x\right) $ is empty.


\bigskip

We now derive the adjoints of the operators $\A$, $\B$ and $\D$.

\begin{lemma}
The adjoint of $\A$ in $M^2$ is the operator
$\A^{\ast }:D\left( \A^{\ast }\right)\subset M^2 \rightarrow M^{2} $
defined as
\[
\left\{
\begin{array}{c}
D\left( \A^{\ast }\right) =\left\{ \left( y_{0},y_{1}\left( \cdot\right) \right)
\in M^{2}: \, y_{1}\left( \cdot\right) \in W^{1,2}\left( \left[ -\tau ,0\right] ;\mathbb{R}\right)
\hbox{ and }
\, y_1(0)=0 \right\}
\\
\\
\A^*\left( y_{0},y_{1}\left( \cdot\right) \right)=\left( \left( A-\delta
\right) y_{0},s\mapsto \frac{dy_{1}\left( s\right) }{ds}\right)
\end{array}
\right.
\]
\end{lemma}

\begin{proof}
Take $x\in D(\A)$ and $y\in M^2$.
We have
$$\left\langle \A x,y \right\rangle _{M^{2}} =\left( A-\delta
\right) x_{0}y_{0}-\int_{-\tau }^{0}x_{1}^{\prime }\left( s\right)
y_{1}\left( s\right) ds
$$
$$
=\left( A-\delta \right) x_{0}y_{0}-x_{1}\left( 0\right) y_{1}\left(
0\right) +x_{1}\left( -\tau \right) y_{1}\left( -\tau \right) +\int_{-\tau
}^{0}x_{1}\left( s\right) y_{1}^{\prime }\left( s\right) ds
$$
It follows that the set of all $y\in M^2$ such that
$x \to \left\langle \A x,y \right\rangle _{M^{2}}$
can be extended to a linear continuous functional on $M^2$ is given exactly
by $D(\A^*)$.
Then we have, for $x\in D(\A)$ and $y\in D(\A^*)$,
  $$
 \left\langle \A x,y \right\rangle _{M^{2}}
 =\left( A-\delta \right) x_{0}y_{0} +\int_{-\tau
}^{0}x_{1}\left( s\right) y_{1}^{\prime }\left( s\right) ds
$$
and the claim follows.
\end{proof}
%

\begin{lemma}
The adjoint of $\B$ is
$$\B^{\ast }: M^{2}\rightarrow\mathbb{R}, \qquad
\B^{\ast }\left( y_{0},y_{1}\left( .\right) \right)
=-y_{0}+\varepsilon \int_{-\tau }^{0}e^{\eta s}y_{1}\left( s\right) ds.
$$
Moreover the adjoint of $\D$ is
$$\D^{\ast }:\mathbb{R}
\rightarrow \R \times [C([-\tau,0];\R)]^*, \qquad
\D^{\ast }c=c\left( 0,\delta_{0}\right)
$$
where $\delta_0$ is the Dirac's $\delta$ at the point $t=0$.
\end{lemma}

\begin{proof}
We have
\[
\left\langle \B c,\left( y_{0},y_{1}\left( \cdot\right) \right) \right\rangle
_{_{M^{2}}}=c\left( -y_{0}+\varepsilon \int_{-\tau }^{0}e^{\eta
s}y_{1}\left( s\right) ds\right).
\]
Moreover
\[
\left\langle \D x,c\right\rangle _{_{\mathbb{R}}}=cx_{1}\left( 0\right) =c\left( 0\cdot x_{0}+\delta_{0}x_{1}\right)
\]
and the claim follows.
\end{proof}

\bigskip

\subsection{The HJB equation and its explicit solution}

The Current Value Hamiltonian $H_{CV}$ of our problem is
a real valued function defined on the set
$$
E \subset M^{2}\times M^{2}\times \mathbb{R}, \qquad
E=\left\{ \left( x,p,c\right) \in D\left( A\right) \times M^{2}\times \R
\right\}
$$
and is given by
\begin{equation}\label{eq:HCV}
H_{CV}\left( x,p\,;c\right)  =\frac{(c-\D x)^{1-\gamma }}{1-\gamma }
+\left\langle \A x,p\right\rangle _{M^{2}}+\left\langle \B^{\ast
}p,c\right\rangle _{\mathbb{R}}
\end{equation}
When $\gamma >1,$ $H_{CV}\left( x,p\,;c\right) $ is not defined in the points
such that $c=x_{1}\left( 0\right) .$ In such points, since the utility is $- \infty$, we set
$H_{CV}\left( x,p\,;c\right) =-\infty $.
The maximum value of the Hamiltonian is defined by
$\mathcal{H}\left(
x,p\right) =\sup_{c\ge x_1(0)}H_{CV}\left( x,p\, ;c\right)$.
The HJB equation of the problem is then
\begin{equation}\label{eq:HJB}
\rho v\left( x\right) -\mathcal{H}\left( x,Dv\left( x\right) \right) =0
\end{equation}
where the unknown $v$ ``should'' be the value function $V_0$.
We will use the following definition of solution.
Note that it is different from the one used in the papers \cite{bambi2010,FaGo}.
\begin{definition}\label{df:solHJB}
We say that a function $v$ is a classical solution of the HJB equation (\ref{eq:HJB}) in an open set
$\mathcal{Y}\subseteq M^2$ if it is differentiable at every $x \in \mathcal{Y}$
and if satisfies  (\ref{eq:HJB}) in every point of $\mathcal{Y} \cap D(\A )$.
\end{definition}

Now we find a solution of the HJB equation. First we compute the maximum value Hamiltonian in the following lemma, whose proof is immediate.

\begin{lemma}\label{lm:maxham}
Given any $p \in M^2$  such that $\B^* p <0$ and any
$x\in D\left( \A\right)$, the function
$$H_{CV}\left( x,p\,; \cdot\right):[x_{1}\left( 0\right) ,\infty )\rightarrow
\mathbb{R}$$
admits a unique maximum point
\[
c^{\max }=
\D x+\left( -\B^{\ast }p\right) ^{-1/\gamma }.
\]
So, in this case
\[
\mathcal{H}\left( x,p\right) =\left\langle \A x,p\right\rangle _{M^{2}}
+\frac{\gamma }{1-\gamma }(-\B^{\ast }p)^{\frac{\gamma -1}{\gamma }}
+\left\langle \D x,\B^{\ast }p\right\rangle _{\R}.
\]
If, on the other hand, $\B^* p \ge0$, then
$$
\sup_{c\ge x_1(0)}H_{CV}\left( x,p\, ;c\right) =+\infty.
$$
\end{lemma}

We now define, for $x \in M^2$,
\begin{eqnarray*}
G\left( x\right)  &=&\left( 1-\varepsilon \int_{-\tau }^{0}e^{\left(
A-\delta +\eta \right) s}ds\right) x_{0}-\int_{-\tau }^{0}e^{\left( A-\delta
\right) s}x_{1}\left( s\right) ds=\left\langle x,\kappa \right\rangle  \\
\text{ where }\kappa  &=&\left( 1-\varepsilon \int_{-\tau }^{0}e^{\left(
A-\delta +\eta \right) s}ds,s\mapsto -e^{\left( A-\delta \right)
s}\right).
\end{eqnarray*}
It is worth noting that $\kappa_0>0$
when we assume (\ref{eq:hpgrowthfinal}) i.e.
that $A-\delta>0\ge\eps-\eta$.
In fact looking at $\kappa_0$ as function of $\tau$ we see that
its derivative with respects to $\tau$ is always negative.
Since it converges to $\frac{A-\delta+\eta-\eps}{A-\delta+\eta}>0$
when $\tau\rightarrow +\infty$, it must be always positive.

We call $\mathcal{X}$ the open subset of $M^{2}$ defined by
\[
\mathcal{X}=\{x=\left( x_{0},x_{1}\left( \cdot\right) \right) \in
M^{2},\text{ }G(x)>0\}.
\]
\begin{proposition}
\label{prop:HJBv} The function $v\left( x\right) =\nu \left( G\left( x\right)
\right) ^{1-\gamma }$
with
\begin{eqnarray*}
\nu  =\frac{1}{1-\gamma}\left( \frac{\rho -\left( A-\delta \right) \left( 1-\gamma
\right) }{\gamma }\right) ^{-\gamma }
\end{eqnarray*}
is differentiable at all $x\in \mathcal{X}$ and is a
solution of the HJB equation in $\mathcal{X}$.
\end{proposition}

\begin{proof}
Let $v\left( x\right) =\nu \left( G\left( x\right) \right)^{1-\gamma }$
for every $x \in M^2$.
%

%


Then
\[
Dv\left( x\right) =\left( 1-\gamma \right) \nu G\left( x\right)
^{-\gamma}\kappa
\]
Since $\B^* \kappa=-1$ we have
\begin{eqnarray*}
\mathcal{B}^{\ast }Dv\left( x\right) &=&\left( 1-\gamma \right) \nu G\left(
x\right) ^{-\gamma } \mathcal{B}^{\ast }\kappa
= - \left( 1-\gamma \right) \nu G\left(
x\right) ^{-\gamma }
\\
\left\langle \D x,\B^{\ast }Dv(x)\right\rangle _{\R}&=&
- x_1(0)
\left( 1-\gamma \right) \nu G\left(x\right) ^{-\gamma }
\end{eqnarray*}%
Now for $x\in D\left( \mathcal{A}\right)$ we have
$$
\left\langle \mathcal{A}x,Dv\left( x\right) \right\rangle _{M^{2}}
=\left(
1-\gamma \right) \nu G\left( x\right) ^{-\gamma }
\left\langle \mathcal{A}x,\kappa \right\rangle _{M^{2}}.
$$
Moreover by definition of $\A$ and $\kappa$ we have
(integrating by parts and using that $x(-\tau)=0$ since
$x \in  D\left( \mathcal{A}\right) $)
\begin{eqnarray*}
\left\langle \mathcal{A}x,\kappa \right\rangle _{M^{2}} &=&
\left( A-\delta \right) \left( 1-\varepsilon \int_{-\tau }^{0}e^{\left(
A-\delta +\eta \right) s}ds\right) x_{0}
+\int_{-\tau }^{0}x_{1}^{\prime }\left( s\right) e^{\left( A-\delta \right)
s}ds
 \\
&=&
\left( A-\delta \right) \left( 1-\varepsilon \int_{-\tau }^{0}e^{\left(
A-\delta +\eta \right) s}ds\right) x_{0}
+x_{1}\left( 0\right) -\left( A-\delta \right) \int_{-\tau }^{0}x_{1}\left(
s\right) e^{\left( A-\delta \right) s}ds
\\[2mm]
&=&
\left( A-\delta \right)\left\langle x,\kappa \right\rangle _{M^{2}}
+x_{1}\left( 0\right).
\end{eqnarray*}
It follows that
$$
\mathcal{H}(x,Dv(x))=
\left(
1-\gamma \right) \nu G\left( x\right) ^{-\gamma }\left[
\left\langle \mathcal{A}x,\kappa \right\rangle _{M^{2}}
- x_1(0)\right]
+\frac{\gamma}{1-\gamma}[(1-\gamma)\nu]^{\frac{\gamma-1}{\gamma}}
G(x)^{1-\gamma}
$$
$$
=(A-\delta)\left(1-\gamma \right) \nu G\left( x\right) ^{1-\gamma }
+\frac{\gamma}{1-\gamma}[(1-\gamma)\nu]^{\frac{\gamma-1}{\gamma}}
G(x)^{1-\gamma}
=
$$
$$
=
\nu  G(x)^{1-\gamma}
\left[(A-\delta)(1-\gamma) + \gamma [(1-\gamma)\nu]^{-\frac{1}{\gamma}}\right]
$$
We can now substitute all the above in the HJB equation getting
$$
\rho v\left( x\right) -\mathcal{H}(x,Dv(x))=
$$
$$
=\nu G\left( x\right)^{1-\gamma}
\left[ \rho -(A-\delta)(1-\gamma) - \gamma [(1-\gamma)\nu]^{-\frac{1}{\gamma}}\right]
$$
and the claim follows by the definition of $\nu$.
\end{proof}

\medskip

The optimal feedback policy associated to the above solution
of the HJB equation (\ref{eq:HJB}) is easily found by
Lemma \ref{lm:maxham} and is
\begin{equation}\label{eq:feedback0}
\varphi(x)=x_{1}\left( 0\right) +\alpha G\left( x\right) ,\text{ for }
x\in \mathcal{X}
    \end{equation}
where $\alpha = \frac{\rho -\left( A-\delta \right) \left( 1-\gamma
\right) }{\gamma }$. Observe that $\alpha >0$ thanks to
assumption (\ref{eq:Vfinite}).


\subsection{Closed loop policy}

We need to determine a set of admissible initial data included in $\mathcal{X}$ such that the candidate optimal feedback $\varphi$  given in (\ref{eq:feedback0})
is really optimal. For any $x$ in this set we will have that
$v(x)=V_0(x)$.

We call $C\left( M^{2}\right) $ the set of continuous functions from $M^{2}$
to $\mathbb{R}$. As in Bambi et al. \cite{bambi2010}, we give definitions concerning
feedback strategies.

\begin{definition}
Given an initial condition $q\in M^{2},$ we call $\psi \in C\left(
M^{2}\right) $ a feedback strategy related to $q$ if the equation
\begin{equation}
\left\{
\begin{array}{c}
\frac{dX\left( t\right) }{dt}=\A X\left( t\right) +\B\left( \psi \left(
X\left( t\right) \right) \right) \\
X\left( 0\right) =q
\end{array}
\right.  \label{eq:eqFBS}
\end{equation}
has a unique solution $X_{\psi }\left( t\right) $ in $\Pi =\left\{ f\in
C\left( [0,\infty ),M^{2}\right) ,\frac{df}{dt}\in L_{loc}^{2}\left(
[0,\infty ),D\left( \A\right) ^{\prime }\right) \right\} .$ The set of
feedback strategies related to $q$ is denoted $FS_{q}.$
\end{definition}

\begin{definition}
Given an initial condition $q\in M^{2},$ and $\psi \in FS_{q},$ we say
that $\psi $ is an admissible strategy if the unique solution $X_{\psi
}\left( t\right) $ of (\ref{eq:eqFBS}) satisfies $\psi \left( X_{\psi
}\left( \cdot\right) \right) \in \C_{ad}\left( q\right)  $. We denote $AFS_{q}$ the
set of admissible feedback strategies related to $q$.
\end{definition}

\begin{definition}
We say that $\psi $ is an optimal feedback strategy related to $q$ if
\[
V\left( q\right) =\int_{0}^{\infty }
\frac{(\psi \left( X_{\psi}(t)\right)
-\D X_\psi (t))^{1-\gamma }}{1-\gamma }e^{-\rho t}dt
\]
We denote $OFS_{q}$ the set of optimal feedback strategies related to $q$.
\end{definition}

\bigskip

We first prove that our candidate is always in $FS_{q}$.

\begin{lemma}\label{lm:CLE}
For every $q\in M^{2},$ the map
\begin{eqnarray*}
\varphi  : M^{2}\rightarrow
\mathbb{R}, \qquad
\varphi \left( x\right)  = x_{1}\left( 0\right) +\alpha G\left( x\right) ,
\end{eqnarray*}
is in $FS_{q}$.
\end{lemma}

\begin{proof}
We have to prove that
\begin{equation}
\left\{
\begin{array}{c}
\frac{dX\left( t\right) }{dt}=\mathcal{A}X\left( t\right) +\mathcal{B}\left(
\varphi \left( X\left( t\right) \right) \right)  \\
X\left( 0\right) =q%
\end{array}%
\right.   \label{eq:CLEphi}
\end{equation}
has a unique solution in $\Pi $.
%
We first consider the following functional equation in
$\widetilde{c}$ and $\widetilde{k}$
\[
\left\{
\begin{array}{l}
\widetilde{c}\left( t\right) =\varepsilon \int_{-\tau }^{0}\widetilde{c}%
\left( t+u\right) e^{\eta u}du \\
\qquad \qquad +\alpha \left[ \left( 1-\varepsilon \int_{-\tau }^{0}e^{\left( A-\delta
+\eta \right) s}ds\right) \widetilde{k}\left( t\right) +\int_{-\tau
}^{0}e^{\left( A-\delta \right) s}\varepsilon \int_{-\tau }^{s}\widetilde{c}%
\left( t+u-s\right) e^{\eta u}du\text{ }ds\right]  \\
{\dot{\widetilde{k}}}\left( t\right) =\left( A-\delta \right)
\widetilde{k}\left( t\right) -\widetilde{c}\left( t\right)  \\
\widetilde{c}\left( s\right) =c\left( s\right) \text{ for }s\in \lbrack
-\tau ,0) \\
c\left( 0\right) =\varepsilon \int_{-\tau }^{0}\widetilde{c}\left(
u-s\right) e^{\eta u}du
\\
\qquad \qquad+\alpha \left[ \left( 1-\varepsilon \int_{-\tau
}^{0}e^{\left( A-\delta +\eta \right) s}ds\right) k\left( 0\right)
+\int_{-\tau }^{0}e^{\left( A-\delta \right) s}\varepsilon \int_{-\tau
}^{s}c\left( u-s\right) e^{\eta u}du\text{ }ds\right] >0 \\
\widetilde{k}\left( 0\right) =k\left( 0\right)
\end{array}%
\right.
\]
This system has a unique continuous solution $\left( \widetilde{c},%
\widetilde{k}\right) $ on $[0,\infty )$ (d'Albis et al., \cite{hupkes2012}).
Denoting $\widetilde{x}=\left( \widetilde{k},\widetilde{\gamma }\left(
t\right) \right) $ where $\widetilde{\gamma }\left( t\right) [s]=\varepsilon
\int_{-\tau }^{s}\widetilde{c}\left( t+u-s\right) e^{\eta u}du,$ then $%
\widetilde{x}$ satisfies
\[
\left\{
\begin{array}{l}
\frac{d\widetilde{x}\left( t\right) }{dt}=\mathcal{A}\widetilde{x}\left(
t\right) +\mathcal{B}\widetilde{c}\left( t\right)  \\
\widetilde{x}\left( 0\right) =\left( k_{0},\widetilde{\gamma }\left(
0\right) \right)
\end{array}%
\right.
\]
which has a unique solution using e.g. Bensoussan et al.
(\cite{bensousan92}, Theorem 5.1, p.282). Notice that $\widetilde{c}\left(
t\right) =\varphi \left( \widetilde{x}\left( t\right) \right)$.

In this way we have proved existence and uniqueness when the
initial datum is of the form
$\left( k_{0},\widetilde{\gamma }\left(0\right) \right)$.
To get the result for every initial datum $q \in M^2$ we need to
set the equation in the space $D(\A)'$ and then show that the solution is
indeed continuous with values in $M^2$. This can be done exactly as in
\cite{FagGozJME}, Section 5-6. We do not do it for brevity and also since,
to solve our starting problem ({\bf P}) it is enough to deal with
the narrower set of data used here.
\end{proof}

\medskip

Now we want to prove the optimality of $\varphi$.
This is very difficult to prove (and in general not true)
without additional assumptions.
So we will prove the optimality of $\varphi$ when
(\ref{eq:hpgrowthfinal}) holds and the initial datum $q$ belongs to a given set $I \subset {\mathcal X}$ which includes the data we are interested in. We start by proving a useful invariance property for the trajectory associated to $\varphi$.
\begin{proposition}
\label{pr:Xinvariant}
For every initial datum $q \in M^2$
the solution $X_{\varphi }\left( \cdot\right) $ of (\ref{eq:CLEphi}) satisfies
\[
G\left( X_{\varphi }\left( t\right) \right)
=G\left( q\right)  e^{^{\Gamma t}}\text{ for all }t\geq 0
\]
where $\Gamma =\frac{1}{\gamma }(A-\delta -\rho )$.
\end{proposition}

\begin{proof}
It is enough to compute $\frac{d}{dt}G\left( X_{\varphi }\left( t\right) \right) $. Indeed we have
\begin{eqnarray*}
\frac{d}{dt}G\left( X_{\varphi }\left( t\right) \right)  &=&\frac{d}{dt}%
\left\langle X_{\varphi }\left( t\right) ,\kappa \right\rangle  \\[2mm]
&=&\left\langle \A X_{\varphi }\left( t\right) +\B \varphi \left( X_{\varphi
}\left( t\right) \right) ,\kappa \right\rangle
\end{eqnarray*}
Now we cannot, as done in other papers (see e.g. \cite{bambi2010,FaGo}, write
$$\left\langle \A X_{\varphi }\left( t\right) ,\kappa \right\rangle
=\left\langle X_{\varphi }\left( t\right) ,\A^{\ast }\kappa \right\rangle
$$
So we have to compute such term directly. Since
$$\A\left( x_{0},x_{1}\left( \cdot\right) \right) = \left( \left( A-\delta
\right) x_{0},s\mapsto -\frac{dx_{1}\left( s\right) }{ds}\right)
$$
and
$$
\kappa  =\left( 1-\varepsilon \int_{-\tau }^{0}e^{\left(
A-\delta +\eta \right) s}ds,s\mapsto -e^{\left( A-\delta \right)s}\right),
$$
then, integrating by parts as in the proof of Proposition \ref{prop:HJBv},
$$
\left\langle \A X_{\varphi }\left( t\right) ,\kappa \right\rangle
= \left( A-\delta\right)X_{\varphi,0}(t)  \left( 1-\varepsilon \int_{-\tau }^{0}e^{\left(
A-\delta +\eta \right) s}ds\right)+ \int_{-\tau}^0 \frac{dX_{\varphi,1}(t)[s]}{ds}
e^{\left( A-\delta \right)s}ds
$$
$$
=X_{\varphi,1}(t)[0] + \left( A-\delta\right)\left\langle X_{\varphi }\left( t\right) ,\kappa \right\rangle
=X_{\varphi,1}(t)[0] + \left( A-\delta\right)
G(X_{\varphi }\left( t\right))
$$
Moreover, since $\B c=c\left( -1,s\mapsto \varepsilon e^{\eta s}\right)$,
and
$$
\varphi ( X_{\varphi}(t))=
 X_{\varphi, 1}(t)\left[ 0\right] +\alpha G\left(
X_{\varphi }\left( t\right) \right)
$$
then
$$
\left\langle \B\varphi\left( X_{\varphi}(t)\right) ,\kappa \right\rangle
=
\left\langle \B\left( X_{\varphi, 1}(t)\left[ 0\right] +\alpha G\left(
X_{\varphi }\left( t\right) \right) \right) ,\kappa \right\rangle
=
$$
$$
=
\left( X_{\varphi, 1}(t)\left[ 0\right]
+\alpha G\left(X_{\varphi }\left( t\right) \right)
\right)
\left(-1+\varepsilon \int_{-\tau }^{0}e^{\left(
A-\delta +\eta \right) s}ds - \varepsilon\int_{-\tau}^0  e^{(A - \delta+\eta) s}ds
\right)
$$
$$
=- X_{\varphi, 1}(t)\left[ 0\right]
-\alpha G\left(X_{\varphi }\left( t\right) \right)
$$
Hence, summing up, we get
$$
\frac{d}{dt}G\left( X_{\varphi }\left( t\right) \right)
=\left( A-\delta -\alpha \right) G\left( X_{\varphi }\left( t\right)
\right).
$$
Using that
\[
A-\delta -\alpha =\Gamma
\]
the claim follows.
\end{proof}

\medskip

Now we define the set $I$ and state a key invariance property of it.

\begin{proposition}\label{pr:Iadmissible}
The set $I$ defined as
\[
I=\mathcal{X}\cap \left\{
\begin{array}{c}
q=\left( x_{0},x_{1}\right) \in \R \times W^{1,2}([-\tau,0];\R) \subset M^{2},
\\
x_{1}\left( s\right) >0 \text{ for  all }s\in [ -\tau ,0],
\end{array}
\right\}
\]%
is invariant for the flow of the autonomous ODE
\[
\frac{dX\left( t\right) }{dt}=\mathcal{A}X\left( t\right) +\mathcal{B}\left(
\varphi \left( X\left( t\right) \right) \right).
\]%
Hence, if (\ref{eq:hpgrowthfinal}) holds, then
for any $q\in I$ we have $\varphi \in AFS_{q}.$
\end{proposition}

\begin{proof}
Let $q=(x_0,x_1(\cdot))\in I$.
We show that the associated solution $X_\varphi (t)$
of (\ref{eq:CLEphi}) still belongs to $I$ for every $t>0$.
Since we already know, by Proposition \ref{pr:Xinvariant},
that we always have $G(X_\varphi(t))>0$, it is enough to prove that, for every $t>0$,
$X_{\varphi,1}(t)[0]>0$ and $X_{\varphi,1}(t)[s] >0$
for almost all $s\in \lbrack -\tau ,0)$.


Let now $t_0\ge 0$ be the supremum of all times $t$ such that the above remain true.
We show the $t_0=+\infty$.
First of all observe that, by using the definition of
structural state, we have, for $t\ge 0$ and $s \in [-\tau,0]$,
\begin{equation}\label{eq:Xphinew}
X_{\varphi ,1}\left( t\right) \left[ s\right] =
\left\{
\begin{array}{ll}
x_1 (s-t)
+\varepsilon e^{\eta \left( s-t\right)}
\int_{0}^{t}\bar c\left( u\right) e^{\eta u}du,
& \hbox{ if }t-s-\tau <0,
\\
\\
\varepsilon e^{\eta \left( s-t\right)}
\int_{t-s-\tau }^{t}\bar c\left( u\right) e^{\eta u}du,
&\hbox{ if }t-s-\tau \ge 0,
\end{array}
\right.
\end{equation}
where
\[
\bar c\left( u\right) =X_{\varphi,1}(u)\left[ 0\right] +\alpha G\left(
X_\varphi(u)\right) \quad for \quad 0\le u< t.
\]
Since $G(X_\varphi(t))>0$ for every $t\ge 0$,
from the above is clear that, for small $t>0$ and for every $s\in[-\tau,0]$
it must be  $X_{\varphi,1}(u)\left[ s\right]>0$.
So it must be $t_0>0$.

Now assume by contradiction the $t_0$ is finite.
The we have
\[
\bar c\left( u\right) =X_{\varphi,1}(u)\left[ 0\right] +\alpha G\left(
X_\varphi(u)\right)>0 \quad for \quad 0\le u< t_0.
\]
So according to (\ref{eq:Xphinew}) it must be
$$X_{\varphi,1}\left( t_{0}\right)\left[ s\right]
>0
$$
for every $s\in[-\tau,0]$. This give the contradiction showing
the invariance of $I$ since clearly the $W^{1,2} $ regularity in $s$ preserves
due to (\ref{eq:Xphinew}).

Finally we observe that, if (\ref{eq:hpgrowthfinal}) holds, then
$X_{\varphi ,0}\left( t\right) >0$ for all $t\ge 0$.
Indeed, since
$G\left( X_{\varphi }(t)\right) >0$
we must have
$$\left(1-\varepsilon \int_{-\tau }^{0}e^{\left( A-\delta +\eta \right) s}ds\right)
X_{\varphi ,0}\left( t\right) >\int_{-\tau }^{0}e^{\left( A-\delta
\right) s}X_{\varphi ,1}(t)[s]ds\geq 0.$$
Recalling that assumption (\ref{eq:hpgrowthfinal}) implies that
$1-\varepsilon \int_{-\tau }^{0}e^{\left( A-\delta +\eta \right) s}ds>0$
(see the discussion before Proposition \ref{prop:HJBv})
we immediately get $X_{\varphi ,0}\left( t_{0}\right)
>0$.
Thus $\varphi \in AFS_{p}.$
\end{proof}


\bigskip

It now remains to prove that $\varphi \in OFS_{p}$.

\begin{proposition}\label{pr:Ioptimal}
If $q\in I$ and (\ref{eq:hpgrowthfinal}) holds, then $\varphi \in OFS_{q}$.
The associated state-control couple is the
unique optimal couple of the problem.
\end{proposition}

\begin{proof}
Let us consider the solution of the HJB equation $v\left( x\right) =\nu
\left( G\left( x\right) \right) ^{1-\gamma }$ and the function
\begin{eqnarray*}
\widetilde{v}\left( t,x\right) &:&
\mathbb{R}
\times M^{2}\rightarrow
\mathbb{R}
\\
\widetilde{v}\left( t,x\right) &=&e^{-\rho t}v\left( x\right)
\end{eqnarray*}
Now take $q \in I $, take any admissible control $c(\cdot)\in \C_{ad}(q)$ and
call $X(\cdot)$ the associated state trajectory starting at $q$. Then we have
\begin{eqnarray*}
\frac{d\widetilde{v}\left( t,X\left( t\right) \right) }{dt}
&=&-\rho e^{-\rho t}v\left( X\left( t\right)\right)
+e^{-\rho t}<D{v}\left(X\left( t\right) \right) ,\A X\left( t\right)  +\B c(t) > \\
\end{eqnarray*}
Observe that the above make sense since, by construction (see e.g. (\ref{eq:Xphinew}))
it must be $X(t) \in D(\A)$ when $q \in I$.

Integrating on $\left[ 0,\tau \right] $ yields to%
$$
e^{-\rho \tau}v\left(X\left( \tau \right) \right)
-v\left(X\left( 0\right) \right) =
$$
\begin{equation}
=\int_{0}^{\tau
}e^{-\rho t}\left[
-\rho v\left( X\left( t\right) \right)
+<Dv\left(X\left( t\right) \right),
\A X\left( t\right) >
+<\B^*Dv\left(X\left( t\right) \right), c\left( t\right) >
\right] dt  \label{eq:integrationpartieV}
\end{equation}
Now
$$G\left( X\left( t\right) \right) =\left( 1-\varepsilon \int_{-\tau
}^{0}e^{\left( A-\delta +\eta \right) s}ds\right) X_0\left( t\right)
- \int_{-\tau }^{0}e^{\left( A-\delta \right) s}
X_1(t)[s]ds.$$
Since, as noted before Proposition \ref{prop:HJBv}, we have
$ 1-\varepsilon \int_{-\tau }^{0}e^{\left( A-\delta +\eta \right)
s}ds>0,$
then it must be $G\left( X\left( t\right) \right) \leq \left(
1-\varepsilon \int_{-\tau }^{0}e^{\left( A-\delta +\eta \right) s}ds\right)
X_0\left( t\right)$, thus
\[
e^{-\rho \tau }G\left( X\left( t\right) \right) ^{1-\gamma }\leq \left(
1-\varepsilon \int_{-\tau }^{0}e^{\left( A-\delta +\eta \right) s}ds\right)
^{1-\gamma }e^{-\left( \rho -\left( 1-\gamma \right) \left( A-\delta \right)
\right) \tau }\left( \frac{X_0\left( t\right) }{e^{\left( A-\delta \right)
t}}\right) ^{1-\gamma }
\]
According to Proposition \ref{pr:estimateck} (since clearly
$X_0(t)\le k^M(t)$) we thus have that
$$\lim_{\tau
\rightarrow \infty }e^{-\rho\tau} v\left(X\left( \tau
\right) \right) =0.$$
Hence, using that $q=X\left( 0\right) $ and taking
the limit as $\tau $ tends to infinite in (\ref{eq:integrationpartieV}),
we obtain
\[
-v\left( q\right) =
\]
\begin{equation}
=\int_{0}^{+\infty}e^{-\rho t}\left[
-\rho v\left( X\left( t\right) \right)
+<Dv\left(X\left( t\right) \right),
\A X\left( t\right) >
+<\B^*Dv\left( X\left( t\right) \right)
, c\left( t\right) >
\right] dt
 \label{eq:integrationpartieVbis}
\end{equation}
so using the definition (\ref{eq:HCV}) of current value Hamiltonian
\begin{eqnarray*}
v\left( q\right) -J_0\left( q;c(\cdot)\right)
=\int_{0}^{\infty }e^{-\rho t}\left( \rho v\left( X\left(
t\right) \right) -H_{CV}\left( X\left( t\right),
Dv\left( X\left(t\right) \right) ,c(t)\right) \right) dt
\end{eqnarray*}
As the value function solves $\rho v\left( x\right) -\mathcal{H}\left(
x,Dv\left( x\right) \right) =0,$ the above implies that
\begin{equation}\label{eq:fundid}
v\left( q\right) -J_0\left( q;c(\cdot)\right)
=\int_{0}^{\infty }e^{-\rho t}
\left[\mathcal{H}\left( X\left( t\right),Dv\left( X\left(t\right) \right) \right)
-H_{CV}\left(X\left( t\right) ,
Dv\left( X\left(t\right) \right) ,c(t)\right)\right] dt
\end{equation}
According to the definition of $\mathcal{H}$, for every admissible control
the integrand of the above right hand side is always positive.
This implies, according to the definition of $V_{0}$, that
\[
v\left( q\right) \geq V_{0}\left( q\right)
\]
and this must be true for every $q \in I$.
Moreover, choosing $c(t)=\varphi (X_\varphi (t))$ (which is admissible thanks to Proposition \ref{pr:Iadmissible}) clearly the right
hand side becomes zero and so such control strategy is optimal. This implies that
$v\left( q\right) = V_{0}\left( q\right)$ for every $q \in I$.

Finally, if $c^1(\cdot)$ is another optimal strategy
(with associated state trajectory $X^1(\cdot)$)
it must satisfy (\ref{eq:fundid}) (where now $v=V_0$ since they are equal on $I$).
So it must be necessarily, for a.e $t\ge 0$,
$$
\mathcal{H}\left( X^1\left( t\right),Dv\left( X^1\left(t\right) \right) \right)
-H_{CV}\left(X^1\left( t\right) ,
Dv\left( X^1\left(t\right) \right) ,c^1(t)\right)=0
$$
which implies that, $t$ a.e., $c^1(t)=\varphi (X^1(t))$.
By the uniqueness of the solutions of the closed loop equation
(\ref{eq:CLEphi}) proved in Lemma \ref{lm:CLE}, we then get
that, $t$ a.e., $c_1(t)= c(t)$.
\end{proof}

\bigskip

\begin{remark}\label{rm:technical} To get the explicit solution $v$ of the HJB equation and to find the optimal closed loop policy $\varphi$, we cannot apply directly the approach used in \cite{FaGo} or in \cite{bambi2010} due to the presence of the delay in the constraint and in the objective functional instead than in the state equation. This changes the structure of the problem. In particular, differently from the case treated in the previous papers, the gradient $Dv$ of the solution of the HJB equation does not belong to $D(A^*)$ and so the concept of solution of such equation must be changed (compare Definition \ref{df:solHJB} with the analogous one of such papers). This fact induces a change in the arguments of the main proofs: the fact that $v$ solves the HJB equation and the fact that the feedback strategy $\varphi$ is admissible and optimal. \end{remark}

It is worth noting that the optimality of $\varphi$ depends on the
initial datum $q$ to belong to the set $I$; this implicitly implies a restriction on the initial value of capital which we may choose. This restriction will be made explicit in the next Section after Proposition \ref{pr:equivalence} and its economic meaning will be also explained.

\section{Equivalence\label{sec:equivalence}}

%
%

The result of the previous section gives the optimal feedback map in the infinite dimensional setting. We now use this result to write the closed loop policy formula in
the delay differential equation setting and we use it to prove the equivalence
between the cases of internal and external habits.


\begin{proposition}\label{pr:equivalence}
Given any initial datum $(k_0, c_0(\cdot))$ the problem (P) above has a unique optimal state-control couple $(k^*(\cdot),c^*(\cdot))$.
Such couple is the only one that satisfies the closed-loop formula:
\[
\frac{c\left( t\right) -h\left( t\right) }{ A-\delta -\Gamma  }
=\]
\[=\left( 1-\varepsilon \int_{-\tau }^{0}e^{\left( A-\delta +\eta \right)
s}ds\right) k\left( t\right) -\left[ \frac{h\left( t\right) }{A-\delta +\eta
}-\varepsilon e^{-\left( A-\delta +\eta \right) \tau }\int_{t-\tau }^{t}%
\frac{e^{\left( A-\delta \right) \left( t-s\right) }}{A-\delta +\eta }%
\tilde c\left( s\right) \text{ }ds\right]
\]
where $h(t)$ is given by (\ref{HABITtilde}).
\end{proposition}

\begin{proof}
We have, by the definition of the optimal feedback map $\varphi$, that, on the optimal path,
$$c\left( t\right) -h\left( t\right) = \alpha G\left( X(t)\right)
$$
$$
=\left( A-\delta -\Gamma \right) \left[
\left( 1-\varepsilon \int_{-\tau }^{0}e^{\left( A-\delta +\eta \right)
s}ds\right)X_0(t) -\int_{-\tau }^{0}e^{\left( A-\delta \right)
s}X_1(t)[s]ds \right]
$$
Now we know that $X_0(t)=k(t)$ while $X_1(t)[s]=\eps\int_{-\tau}^{s}
\tilde c(t+u-s) e^{\eta u}du$ so, substituting, we have,
\[
\frac{c\left( t\right) -h\left( t\right) }{A-\delta -\Gamma }
=
\left( 1-\varepsilon \int_{-\tau }^{0}e^{\left( A-\delta +\eta \right)
s}ds\right) k\left( t\right) - \eps
\int_{-\tau }^{0}e^{\left( A-\delta \right)
s}\int_{-\tau}^{s}
\tilde c(t+u-s) e^{\eta u}du ds
\]
Now we integrate by parts obtaining, with straightforward computations,
\begin{eqnarray*}
\int_{-\tau }^{0}e^{\left( A-\delta \right)
s}\int_{-\tau}^{s}
\tilde c(t+u-s) e^{\eta u}du ds
\end{eqnarray*}
$$
=\frac{1}{A-\delta +\eta }\int_{-\tau }^{0}c\left( t+v\right) e^{\eta
v}dv-e^{-\left( A-\delta +\eta \right) \tau }\int_{t-\tau }^{t}\frac{%
e^{\left( A-\delta \right) \left( t-s\right) }}{A-\delta +\eta }c\left(
s\right) \text{ }ds
$$
which gives the claim.
\end{proof}

\bigskip

Using the above result it is not difficult to prove, by straightforward computations, that
given any initial data $\left( k_{0},c_{0}\left( \cdot\right) \right) ,$ there
exists a $\Lambda $ such that, along an optimal trajectory, the optimal
control $c\left( \cdot\right)^* $ satisfies
\[
c(t)-h(t)=\Lambda e^{\Gamma t}
\]%
with
$$
\Lambda =\left( A-\delta -\Gamma \right)\cdot
$$
\begin{equation}\label{eq:Lambda}
\cdot \left( \left( 1-\varepsilon
\int_{-\tau }^{0}e^{\left( A-\delta +\eta \right) s}ds\right) k\left(
0\right) -\left[ \frac{h\left( 0\right) }{A-\delta +\eta }-\varepsilon
e^{-\left( A-\delta +\eta \right) \tau }\int_{-\tau }^{0}\frac{e^{-\left(
A-\delta \right) s}}{A-\delta +\eta }c\left( s\right) \text{ }ds\right]
\right).
\end{equation}
It is worth noting that the constraint, $c(t)\geq h(t)$, is respected if $\Lambda>0$ or equivalently, in term of the initial capital stock, if
\[k(0)\geq \frac{h(0)}{A-\delta+\eta-\varepsilon+\varepsilon e^{-(A-\delta+\eta)\tau}}-\frac{\varepsilon e^{-\left( A-\delta +\eta \right) \tau }}{A-\delta+\eta-\varepsilon+\varepsilon e^{-(A-\delta+\eta)\tau}}\int_{-\tau
}^{0}e^{-\left( A-\delta \right) u}c\left( u\right) du.\]
In the specific case, $\tau=\infty$ and $\eps=\eta$ this condition becomes $rk(0)>h(0)$ meaning that capital income (which in our context coincides with the initial wealth) has to be higher than the initial habits otherwise an initial consumption higher than $h(0)$ will pin down a consumption path not sustainable over time since financed with the resources coming from disinvestments.

In the case with a finite $\tau$ this condition becomes less restrictive as the first term in the right hand side of the inequality becomes smaller and the second negative term appears. The reason is that the stock of habits is now formed over a finite consumption history and therefore less resources are needed at the beginning because the past consumption affecting the habit formation will be completely ``depreciated'' after a period of length $\tau$.

At this point we have all the information for proving the main result of the paper.

\begin{theorem}[Equivalence Theorem]
Consider an economy with subtractive nonseparable C.E.S. utility function and linear technology. Then internal and external habits lead to the same unique equilibrium path. This path is Pareto optimal.
\end{theorem}
\begin{proof}
  The closed loop policy formula for the external case was found in Augeraud-Veron and Bambi \cite{auba} (see proof of Proposition 4) using a modified version of the Pontryagin Maximum Principle. Such function writes:
$$
\frac{c(t)-h(t)}{A-\delta -\Gamma }=k(t)-\frac{1}{A-\delta +\eta }\left[ h\left(
t\right) +\varepsilon \left( 1-e^{-\left( A-\delta +\eta \right) \tau
}\right) k\left( t\right)\right.
$$
\begin{equation}
\left. -\varepsilon e^{-\eta \tau }e^{-\left( A-\delta
\right) \tau }\int_{t-\tau }^{t}e^{-\left( A-\delta \right) \left(
u-t\right) }c\left( u\right) du\right]   \label{eq: extpolicy}
\end{equation}
Therefore the equivalence between external and internal habits emerges immediately comparing it with the result of Proposition \ref{pr:equivalence}. It is worth noting that
the initial value of the costate variable found in \cite{auba} is exactly equal to the constant $\Lambda$ given in (\ref{eq:Lambda}), as expected.
\end{proof}

\bigskip

Therefore the market equilibrium path in the case of external habits has been proved to be Pareto optimal since it coincides with the optimal path derived by solving the problem with internal habits. Interestingly enough this result is robust to any selections of the key parameters in the economy and even more importantly to any specification of the parameter $\tau$ capturing the consumption history relevant in the formation of the habits.

\section{Concluding remarks}

In this paper, we have shown that internal and external habits may lead to the same closed loop policy function and then to the same Pareto optimal equilibrium path. Interestingly enough this anomaly emerges when the utility function is subtractive nonseparable but not for the multiplicative nonseparable case as emerged from Carroll et al. \cite{CarWe1}. Therefore models with habits formation and subtractive utility function needs a higher degree of heterogeneity, as for example different initial endowments across the households, to guarantee a clear distinction between the external and internal specification. Alternatively, the multiplicative nonseparable formulation should be preferred and used.

\end{document}